\documentclass[a4paper,11pt]{amsart}

\usepackage{amsmath}
\usepackage{amssymb}
\usepackage{amsfonts}
\usepackage{graphicx}
\usepackage[percent]{overpic}
\usepackage{float}
\usepackage{color}

\usepackage[margin=1.2in]{geometry}

\newcommand{\mcg}{\mathrm{MCG}}

\newcommand{\card}[1]{\left| #1\right|}
\newcommand{\aut}[1]{\mathrm{Aut}\left( #1\right)}
\newcommand{\sym}[1]{\mathfrak{S}_{#1}}
\newcommand{\baut}[1]{\mathrm{Aut}\left(\partial #1\right)}
\newcommand{\cg}{\mathrm{RC}}

\newtheorem{theor}{Theorem}
\newtheorem{cor}[theor]{Corollary}
\newtheorem{lemma}[theor]{Lemma}
\newtheorem{prp}[theor]{Proposition}
\theoremstyle{definition}

\newtheorem{innerthmrep}{Theorem}
\newenvironment{thmrep}[1]
  {\renewcommand\theinnerthmrep{#1}\innerthmrep}
  {\endinnerthmrep}

\newtheorem{innercorrep}{Corollary}
\newenvironment{correp}[1]
  {\renewcommand\theinnercorrep{#1}\innercorrep}
  {\endinnercorrep}

\title[Mapping class group orbits of curves with self-intersections]{Mapping class group orbits of curves\\with self-intersections}
\author{Patricia Cahn, Federica Fanoni and Bram Petri}
\address{\newline\noindent 
Patricia Cahn\newline
Max Planck Institute for Mathematics, Bonn, Germany\newline
pcahn@mpim-bonn.mpg.de \newline
\newline
Federica Fanoni\newline 
Mathematics Institute, University of Warwick, Coventry, UK\newline
federica.fanoni@gmail.com \newline
\newline
Bram Petri\newline 
Max Planck Institute for Mathematics, Bonn, Germany\newline
brampetri@mpim-bonn.mpg.de}

\date{\today}

\begin{document}

\begin{abstract} We study mapping class group orbits of homotopy and isotopy classes of curves with self-intersections. We exhibit the asymptotics of the number of such orbits of curves with a bounded number of self-intersections, as the complexity of the surface tends to infinity.

We also consider the minimal genus of a subsurface that contains the curve. We determine the asymptotic number of orbits of curves with a fixed minimal genus and a bounded self-intersection number, as the complexity of the surface tends to infinity. 

As a corollary of our methods, we obtain that most curves that are homotopic are also isotopic. Furthermore, using a theorem by Basmajian, we get a bound on the number of mapping class group orbits on a given a hyperbolic surface that can contain short curves. For a fixed length, this bound is polynomial in the signature of the surface.

The arguments we use are based on counting embeddings of ribbon graphs. 
\end{abstract}

\maketitle

\section{Introduction}

Recently, there has been a lot of progress on counting curves on surfaces. There are essentially two questions to answer, a geometric and a topological one. The topological question asks how many curves there are with given topological properties. The geometric question asks how many closed geodesics, possibly with certain topological properties, there are up to a certain length. We start with a brief and incomplete overview of the work that has been done on these questions. For simplicity, we will for now restrict to closed surfaces.

The classical geometric result is by Huber \cite{Huber} and Margulis \cite{Margulis} and states that given a negatively curved, complete and finite volume metric on a closed surface $\Sigma_g$ of genus $g\geq 2$, the number $G(L)$ of closed geodesics up to length $L>0$ satisfies
$$
G(L) \sim \frac{e^{\delta L}}{\delta L} \;\;\text{as}\;\; L\to\infty
$$
where $\delta$ is the topological entropy of the geodesic flow and the symbol `$\sim$' means that the ratio of the two quantities tends to $1$ as $L\to\infty$. If the metric is hyperbolic (constant curvature $-1$) then $\delta=1$, which is the case that Huber considered. Note that on a surface with a negatively curved metric, closed geodesics naturally correspond to free homotopy classes of non-contractible and non-peripheral (i.e. not homotopic to a single puncture) curves (see for instance \cite[Prop.\ 1.3]{FarbMargalit} or \cite[Theorem 3.8.14]{Klingenberg}), which means that counting closed geodesics is the same as counting free homotopy classes.

In \cite{Mirzakhani1}, Mirzakhani showed that the number $S(L)$ of {\it simple} closed geodesics (closed geodesics with no self-intersections) up to length $L>0$ on a closed hyperbolic surface $X$ of genus $g$ satisfies
$$
S(L) \sim b_X L^{6g-6}\;\;\text{as}\;\; L\to\infty
$$
where $b_X$ is a continuous and proper function on the moduli space of hyperbolic structures on $\Sigma_g$. Earlier results were obtained by Rees \cite{Rees}, McShane and Rivin \cite{McShaneRivin} and Rivin \cite{Rivin1}. Part of the proof of this result relies on dividing the geodesics in to mapping class group (denoted $\mcg(\Sigma_g)$) orbits and then counting the number of curves in a fixed orbit. 

Even more recently the asymptotics of the number of closed geodesics with a bounded number of self-intersections in a given $\mcg(\Sigma_g)$-orbit up to a given length have been shown to behave similarly by Rivin in \cite{Rivin2} (for one self-intersection), Erlandsson and Souto \cite{ErlandssonSouto} and Mirzakhani \cite{Mirzakhani2}.

The topological question asks how many $\mcg(\Sigma_g)$-orbits, or {\it topological types}, there are of (isotopy or homotopy) classes of curves (or sets of curves) with certain properties. Note that in some cases this count is also necessary to complete the geometric picture.

It is not hard to see that there are infinitely many $\mcg(\Sigma_g)$-orbits of curves on $\Sigma_g$. So we need to consider smaller sets of curves. The classical topological result (see for instance \cite[Section 1.3.1]{FarbMargalit}) is that the number $N_{g}(0)$ of $\mcg(\Sigma_g)$-orbits of homotopy classes of simple curves is equal to
$$
N_{g}(0) = \left\lfloor\frac{g}{2}\right\rfloor +1
$$
where $\lfloor x\rfloor$ denotes the floor of a real number $x$ .

One can similarly count $N_{g}(\leq k)$: the number of $\mcg(\Sigma_g)$-orbits of homotopy classes of curves with at most $k$ self-intersections. This question is considerably more difficult than the simple case. For one thing, in this count it actually matters whether one considers isotopy or homotopy classes. This is because, unlike the case of simple curves, these are no longer the same for curves with self-intersections. In order to complete the count of the number of geodesics on a hyperbolic surface with a bounded number of self-intersections up to a bounded length, Sapir \cite{Sapir} considered the asymptotics of $N_{g}(\leq k)$ for $g$ fixed and $k\to\infty$. She proved that
$$
\frac{1}{12}2^{\sqrt{\frac{k}{12}}}\leq N_{g}(\leq k) \leq e^{d_g\sqrt{k}\log(d_g\sqrt{k})}
$$
where $d_g$ is a constant depending only on the genus. Sapir used these results to answer questions on the number of $\mcg(\Sigma_g)$-orbits that contain short curves. Concretely, let $L>0$, $X$ be a hyperbolic surface and $N_X(k,L)$ denote the number of $\mcg(\Sigma_g)$-orbits of curves with $k$ self-intersections the contain a curve of length at most $L$. She proves
$$
\frac{1}{12}\min\left\{2^{\frac{1}{8l_X}},2^{\sqrt{\frac{k}{12}}}\right\} \leq N_X(k,L) \leq \min\left\{e^{d_g\sqrt{k}\log(d_X\frac{L}{\sqrt{k}}+d_X)}, e^{d_g\sqrt{k}\log(d_g\sqrt{k})} \right\}
$$
where $l_X$ and $d_X$ are constants depending only on $X$.

Another question of a similar flavor is about complete $1$-systems, i.e. collections of isotopy classes of simple curves that pairwise intersect exactly once. In \cite{MalesteinRivinTheran}, Malestein, Rivin and Theran raised the question of how many $\mcg(\Sigma_g)$-orbits there are of such systems and showed that there is only one such orbit when $g=1,2$. Aougab \cite{Aougab} and subsequently Aougab and Gaster \cite{AougabGaster} showed that this does not persist in the higher genus case by constructing many such orbits.

The question we ask is complementary to the one considered by Sapir. Instead of asking how many curves there are with a large number of self-intersections on a fixed surface, we ask how many curves there are with a fixed number of self-intersections on a surface of large genus or with a large number of punctures. Besides the number of self-intersections, we also order our orbits of curves by the minimal genus of a subsurface that contains them.

We have already noted that it makes a difference whether one asks for homotopy classes or isotopy classes of curves. Let us start with homotopy classes.

Concretely, let $N_{g,n}(k,h)$ denote the number of $\mcg(\Sigma_{g,n})$-orbits of free homotopy classes of curves on $\Sigma_{g,n}$, a surface of signature $(g,n)$, that have $k$ self-intersections and minimal genus of a subsurface containing them equal to $h$ . We prove:
\begin{theor}\label{maintheorem} Let $k,h\in\mathbb{N}$. Furthermore, let $\{g_i,n_i\}_{i\in\mathbb{N}}\subset \mathbb{N}$ be a sequence such that $g_i+n_i\to\infty$ as $i\to\infty$. Then
$$
 N_{g_i,n_i}(k,h) \sim C_{k,h} {g_i+k-3h+1 \choose k+1-2h} {n_i+k+1-2h \choose k+1-2h} 
$$
as $i\to\infty$. Here, $C_{k,h}= \sum\limits_{\Gamma \in \cg_h(k)} \frac{1}{\card{\baut{\Gamma}}}$ is a constant depending only on $k$ and $h$. The sum in $C_{k,h}$ is taken over certain ribbon graphs (see Section \ref{sec_asymptotics}).
\end{theor}

To our knowledge, there is no formula for $C_{k,h}$ that eliminates the dependence on ribbon graphs. On the other hand, similar quantities have been counted, often in terms of chord diagrams (see for instance \cite{Stoimenow}). These counts do give upper bounds for $C_{k,h}$, but they are not sharp. The problem here is the automorphism group that appears in the terms. We also note that in her very recent work \cite[Section 1.7]{Sapir}, Sapir also suggested using cut-and-paste techniques to count mapping class group orbits.  She suggests applying these techniques to small values of $k$ but does not work out the asymptotics.

Theorem \ref{maintheorem} can be used to determine the asymptotics of the number of all orbits of curves with $k$ self-intersections (so without restrictions on their minimal genus). Let us denote this number by $N_{g,n}(k)$. 

\begin{cor}\label{cor_nogenus}
 Let $k\in\mathbb{N}$. Furthermore, let $\{g_i,n_i\}_{i\in\mathbb{N}}\subset \mathbb{N}$ be a sequence such that $g_i+n_i\to\infty$ as $i\to\infty$. 
$$
 N_{g_i,n_i}(k) \sim C_{k}  {g_i+k+1 \choose k+1} {n_i+k+1 \choose k+1} 
$$
as $i\to\infty$. $C_k$ again is a sum over certain ribbon graphs (see Section \ref{sec_asymptotics}).
\end{cor}

If we specialize even further and consider closed surfaces only, we obtain that
$$
N_g(k)\sim C_k \frac{g^{k+1}}{(k+1)!}
$$
as $g\to\infty$.

On the other hand, letting go of the restrictions on the number of self-intersections does not lead to interesting counts: the number of orbits of curves of a given minimal genus (without restrictions on the self-intersection number) is easily seen to be infinite.

It also follows from our arguments that, if $N^{iso}_{g,n}(k,h)$ denotes the number of free {\it isotopy} classes of essential closed curves with $k$ self-intersections and minimal genus $h$, we have:
$$
N^{iso}_{g,n}(k,h) \sim N_{g,n}(k,h)
$$
for $k,h$ fixed and $g+n\to\infty$. In other words, asymptotically it doesn't matter whether one counts orbits of istopy classes or homotopy classes. In particular, we have the follwing:
\begin{cor}\label{cor_prob1} Let $k\in\mathbb{N}$. Furthermore, let $\{g_i,n_i\}_{i\in\mathbb{N}}\subset \mathbb{N}$ be a sequence such that $g_i+n_i\to\infty$ as $i\to\infty$. Then as $i\to\infty$
$$\mathbb{P}_{g_i,n_i}\left[\substack{\displaystyle{\text{The homotopy class of a curve with }k\text{ self}} \\ \displaystyle{\text{intersections contains exactly one isotopy class}}} \right] \to 1$$
\end{cor}

$\mathbb{P}_{g,n}$ here denotes the natural (uniform) probability measure on the finite set of $\mcg(\Sigma_{g,n})$-orbits of curves with $k$ self-intersections. Baer's classical theorem \cite{Baer} says that for simple curves isotopy and homotopy are the same, so Corollary \ref{cor_prob1} can be seen as a probabilistic version of this result.

Our arguments also imply that the asymptotics of $N_{g,n}(\leq k)$ and the similarly defined $N^{iso}_{g,n}(\leq k)$ are dominated by $N_{g,n}(k)$ and $N^{iso}_{g,n}(k)$ respectively as $g+n\to\infty$ and hence that $N^{iso}_{g,n}(\leq k)\sim N_{g,n}(\leq k)\sim N_{g,n}(k)$ as $g+n\to\infty$.

Along the way we also prove the following:
\begin{cor}\label{cor_prob2} Let $k\in\mathbb{N}$. Furthermore, let $\{g_i,n_i\}_{i\in\mathbb{N}}\subset \mathbb{N}$ be a sequence such that $g_i+n_i\to\infty$ as $i\to\infty$. Then as $i\to\infty$
\[ \mathbb{P}_{g_i,n_i}\left[\substack{\displaystyle{\text{A curve with }k\text{ self-intersections has a disk}} \\ \displaystyle{\text{ in the complement}}} \right]\to 0 \]
and
\[ \mathbb{P}_{g_i,n_i}\left[\substack{\displaystyle{\text{A curve with }k\text{ self-intersections separates }\Sigma_g\text{ into}} \\ \displaystyle{k+2\text{ surfaces, all of different signatures}}} \right]\to 1 \]
\end{cor}

Theorem \ref{maintheorem} also has geometric consequences, for which we need a result by Basmajian from \cite{Basmajian}. He proves that a geodesic with $k$ self-intersections has length bounded below by a function of $k$ and of the hyperbolic structure (see Theorem \ref{thm_Basmajian} in Section \ref{sec_geometry}).

Given $L>0$ and a hyperbolic surface $X$ of signature $(g(X),n(X))$, we will write $N_X(L)$ for the number of $\mcg(\Sigma_{g(X)},n(X))$-orbits of closed geodesics on $X$ that contain a curve of length at most $L$. Note that
$$N_X(L)=\sum_{k\geq 0}N_X(k,L),$$
where the $N_X(k,L)$ are the earlier mentioned counts considered by Sapir in \cite{Sapir}.
From Basmajian's bounds we obtain:
\begin{cor}\label{cor_geometry} Let $L>0$. There exist constants $A=A(L)\in\mathbb{N}$ and $C=C(L)>0$ such that for any hyperbolic surface $X$
$$
N_X(L) \leq C\cdot ((g(X)+1)\cdot (n(X)+1))^A
$$
Furthermore, $A(L)$ can be made explicit (see Section \ref{sec_geometry}).
\end{cor}

Because our results are based on counting embeddings of ribbon graphs, we believe that our methods generalize to other sets of bounded numbers of disjoint curves with bounded numbers of self-intersections. Also, by carefully going through the arguments below, the case where $k$ is a moderately growing function of $g+n$ can also be handled. 

\subsection*{Acknowledgement}
The first and third author thank the Max Planck Institute for Mathematics in Bonn for its hospitality. The second author acknowledges support from Swiss National Science Foundation grant number P2FRP2\textunderscore 161723.

\section{Classical results and set up}

\subsection{Curves on surfaces}
In this section we recall some classical theorems about curves on surfaces. First we review results about minimal representatives of curves in a given free homotopy class.

Let $\Sigma_{g,n}$ be a oriented surface of signature $(g,n)$. That is, $\Sigma_{g,n}$ is obtained from an oriented closed surface $\Sigma_g$ of genus $g$ by removing $n$ points. A smooth curve $a:S^1\rightarrow \Sigma_{g,n}$ is said to be {\it generic} if its only singularities are transverse double points.  Define the {\it minimal self-intersection number} $m(\alpha)$\footnote{Some authors write $m(\alpha)=i(\alpha,\alpha)$, where $i$ is the geometric intersection number.} of a free homotopy class of curves $\alpha$ on $\Sigma_{g,n}$ to be the minimum number of double points of a generic representative $a\in \alpha$. If $m(\alpha)=0$, $\alpha$ will be called {\it simple}.

For curves that are not in minimal position, we have the following theorem by Hass and Scott:
\begin{theor}[\cite{HassScottIntersections}, Theorem 2]  Let $a$ be a generic curve on $\Sigma_{g,n}$ which has excess self-intersection. Then there is a singular 1-gon or 2-gon on $\Sigma_{g,n}$ bounded by part of the image of $a$.
\end{theor}

By {\it singular 1-gon} we mean the image by $a$ of an arc $I$ of $S^1$, such that $a$ identifies the endpoints of $I$ and $\left.a\right|_I$ is a null-homotopic loop on $\Sigma_{g,n}$. A {\it singular 2-gon} is the image by $a$ of two disjoint arcs $I$ and $J$ of $S^1$ such that $a$ identifies the endpoints of $I$ and $J$ and $\left.a\right|_{I\cup J}$ is a null-homotopic loop on the surface. Note that these singular 1- or 2-gons do not need to be embedded, but just immersed.

Note also that if $a$ has excess self-intersection, then at least one of the surfaces in the complement of the image of $a$ is homeomorphic to a disk. The converse is not true.

A {\it third Reidemeister move} is a local move which corresponds to pushing a branch of a curve across a double points, as depicted in Figure \ref{reidemeister}.
\begin{figure}[H]
\includegraphics{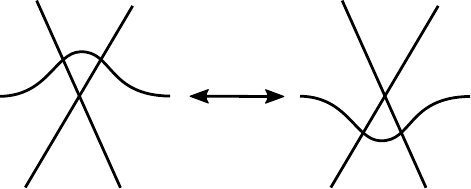}
\caption{A third Reidemeister move}\label{reidemeister}
\end{figure}

If we can perform a third Reidemeister move to a curve $a$, then one of the surfaces in the complement of the image of $a$ is homeomorphic to a disk. Furthermore, if $a$ is in minimal position and $b$ is obtained by $a$ via a third Reidemeister move, then $b$ is in minimal position too. Hass and Scott, and later De Graaf and Schrijver, proved the following:

\begin{theor}[\cite{HassScottShortening}, Theorem 2.1, and \cite{deGraafSchrijver}, Theorem 1]\label{reidemeisterthm}
If $a$ and $b$ are two minimal representatives of the same homotopy class, then there are a sequence of third Reidemeister moves from $a$ to a curve $c$ and an ambient isotopy of $\Sigma_{g,n}$ which induces an isotopy between a regular neighborhood of $c$ and a regular neighborhood of $b$.
\end{theor}

\subsection{Action of the mapping class group on non-simple classes}
The mapping class group of $\Sigma$ is
$$\mcg(\Sigma)=\mbox{Homeo}^+(\Sigma)/_\approx$$
where $\psi\approx\phi$ if they are homotopic. We note that mapping classes in this paper are allowed to permute the punctures of $\Sigma$.

A well known result of Baer \cite{Baer} tells us that simple closed curves are homotopic if and only if they are isotopic (if and only if there is an ambient isotopy of $\Sigma_{g,n}$ sending one to the other). Moreover, we know that two simple closed curves are in the same mapping class group orbit if and only if the surfaces obtained by cutting along them are homeomorphic (see \cite[Section 1.3.1]{FarbMargalit}).

None of these facts hold for nonsimple curves. First, nonsimple curves can be homotopic but not isotopic. (On the other hand Hass and Scott proved in \cite{HassScottConfigurations} that there are only finitely many isotopy classes within a given homotopy class). Second, curves with homeomorphic complements can be in different mapping class group orbits. 

\subsection{Ribbon graphs}\label{gauss&ribbon}
The main tool we will use in our counting arguments later on is {\it ribbon graphs}. Ribbon graphs are graphs together with a {\it (vertex) orientation}. An orientation is a cyclic ordering of the half edges emanating from each vertex of a graph. Here a graph is allowed to have multiple edges and loops\footnote{Some authors prefer the term multigraph for such a graph. We will however not make this distinction.}. We note that writing down a careful definition of ribbon graphs is surprisingly subtle. We will however content ourselves with the description given above and refer to \cite{MulasePenkava} for a rigorous definition.

A ribbon graph can be thickened into a {\it ribbon surface}, that is, an oriented surface with boundary in such a way that the orientation of the surface corresponds to the cyclic orderings of the half edges. In general we will not distinguish between ribbon graphs and ribbon surfaces. We will write $g(\Gamma)$ and $b(\Gamma)$ for the genus and the number of boundary components of $\Gamma$ respectively.

The automorphism group of a ribbon graph $\Gamma$, denoted $\aut{\Gamma}$, is the group of bijective self-maps of $\Gamma$ that preserve the graph structure and the orientation of $\Gamma$. Here, an automorphism is a pair of maps, one that sends vertices to vertices and another one that sends edges to edges. So in particular, an automorphism is allowed to act as the identity on the set of vertices.

Also note that ribbon graph automorphisms extend to orientation preserving homeomorphisms of the corresponding ribbon surface. 

Given a ribbon graph $\Gamma$, we will denote the boundary of the corresponding surface by $\partial\Gamma$. The set of connected components of $\partial\Gamma$ will be denoted $B(\Gamma)=\{\beta_1,\ldots,\beta_{b(\Gamma)}\}$. Note that the restriction of a ribbon graph automorphism to the boundary of the corresponding ribbon surface gives us a map
$$
\aut{\Gamma} \to \sym{B(\Gamma)}
$$
where $\sym{B(\Gamma)}$ denotes the group of permutations of the finite set $B(\Gamma)$, which is isomorphic to the symmetric group $\sym{b(\Gamma)}$ on $b(\Gamma)$ letters. We will denote the image of this map by
$$
\baut{\Gamma} = \mathrm{Im}\left(\aut{\Gamma} \to \sym{B(\Gamma)}\right).
$$

\subsection{Ribbon graphs coming from curves on surfaces}
Our reason to define ribbon graphs is that we can associate them to curves on surfaces. 

Given a generic curve $a$ with self-intersections on $\Sigma_{g,n}$, we can associate a $4$-valent ribbon graph $\Gamma(a)$ to it as follows. The vertices of $\Gamma(a)$ are the self-intersections of $a$ and the edges are the arcs between those self-intersections. The orientation at each vertex comes from the orientation on $\Sigma_{g,n}$. The ribbon surface corresponding to $\Gamma(a)$ is naturally embedded in $\Sigma_{g,n}$ as a regular neighborhood of $a$.

We note however that not all $4$-valent ribbon graphs correspond to a single curve on a surface. We write $\cg(k)$ for the set of isomorphism classes of those that do come from curves with $k$ self-intersections. This set naturally corresponds to the set of so-called Gauss diagams of rank $k$; since we will not directly need to use those in this text we will simply talk about ribbon graphs. For an exposition on Gauss diagrams, see \cite{Turaev}. For enumerative results, see for example \cite{Stoimenow}.

We have the following:
\begin{lemma}\label{isotsamegraph} If $a$ and $b$ are curves on $\Sigma_{g,n}$ in minimal position that can be mapped to each other by ambient isotopies of $\Sigma_{g,n}$ and homeomorphisms of $\Sigma_{g,n}$ then $\Gamma(a)\simeq \Gamma(b)$ as ribbon graphs.
\end{lemma}
\begin{proof}
Ambient isotopies and homeomorphisms send regular neighborhoods to regular neighborhoods, so the ribbon surfaces associated to $a$ and $b$ are isotopic and thus $\Gamma(a)\simeq \Gamma(b)$.
\end{proof}
Note that this lemma implies that the ribbon graph of an isotopy class of curves, defined as the ribbon graph of a minimal representative of the class, is well defined and is an invariant of the mapping class group orbit of such a class.

Because of the existence of third Reidemeister moves, one {\it cannot} uniquely associate a ribbon graph to each homotopy class.  We do however have the following weaker statement:

\begin{lemma}\label{homotsamegraph} If $a$ and $b$ are curves on $\Sigma_{g,n}$ that can be mapped to each other by ambient homotopies of $\Sigma_{g,n}$ and homeomorphisms of $\Sigma_{g,n}$ and furthermore these curves contain no disks in their complement then $\Gamma(a)\simeq \Gamma(b)$ as ribbon graphs.
\end{lemma}
\begin{proof}
Since there are no disks in the complement, we cannot perform third Reidemeister moves. This means that by Theorem \ref{reidemeisterthm} the curves can be mapped to each other by ambient isotopies and homeomorphisms of the surface and we can apply Lemma \ref{isotsamegraph}.
\end{proof}

In order to be able to count mapping class group orbits of homotopy classes of curves, we need a converse to this lemma. We will state everything for isotopy classes first.

To an isotopy class of curves $\alpha$ on $\Sigma_{g,n}$ we associate the triple
$$
V(\alpha) = \left(\Gamma(\alpha),P(\alpha),S(\alpha)\right)
$$
where\begin{itemize}
\item $\Gamma(\alpha)$ is the ribbon graph associated to a minimal representative $a$,
\item $P(\alpha)=\{p_1,\ldots,p_r\}$ is a partition of the set of boundary components of $\Gamma(\alpha)$ such that the boundary components of every $p_i$ form the entire boundary of exactly one connected component of $\Sigma_{g,n}\setminus\Gamma(\alpha)$, and
\item $S(\alpha)=((g_1,n_1,b_1),\ldots,(g_r,n_r,b_r))$ records the signatures of the surface attached to the boundary components in $p_i$ for all $i=1,\ldots,r$.
\end{itemize}
We have already noted that $\Gamma(\alpha)$ is indeed an invariant of the isotopy class of $\alpha$ and the same holds for the partition and the signatures of the surfaces, so the triple is well-defined.

We have the following:
\begin{lemma}\label{lem_orbitclass1} Let $\alpha$ and $\beta$ be free isotopy classes of curves. Then $\alpha$ and $\beta$ lie in the same mapping class group orbit if and only if $\Gamma(\alpha)\simeq\Gamma(\beta)$ and the data $(P(\beta),S(\beta))$ can be obtained from $(P(\alpha),S(\alpha))$ by applying the isomorphism between these graphs to this data.
\end{lemma}
\begin{proof}
If $\alpha$ and $\beta$ lie in the same mapping class group orbit, let $a$ and $b$ be minimal representatives and $\phi$ an orientation preserving homeomorphism sending $a$ to $b$. By Lemma \ref{isotsamegraph}, $\Gamma(a)\simeq \Gamma(b)$. Moreover, it is easy to check that $(P(\beta),S(\beta))$ can be obtained from $(P(\alpha),S(\alpha))$ via $\phi$.

Conversely, suppose $a$ and $b$ are minimal representatives of $\alpha$ and $\beta$ respectively and let $g:\Gamma(\alpha)\rightarrow\Gamma(\beta)$ be the isomorphism given by the hypothesis. This induces an orientation preserving homeomorphism $f$ between the ribbon surfaces of $a$ and $b$, sending $a$ to $b$. The fact that the data $(P(\beta),S(\beta))$ can be obtained from $(P(\alpha),S(\alpha))$ by applying $f$ means that $\Gamma(\alpha)$ and $\Gamma(\beta)$ have homeomorphic complementary components, say $S_1,\dots S_k$ and $S_1',\dots S_k'$, with orientation preserving homeomorphisms $f_i:S_i\rightarrow S_i'$. Moreover, the fact that $P(\beta)$ and $P(\alpha)$ correspond via $g$ implies that we can glue $f,f_1\dots f_k$ to get an orientation preserving homeomorphism of $\Sigma_{g,n}$ sending $a$ to $b$. Note that we need to choose the homeomorphisms $f_i$ to be orientation preserving in order to be able to glue them to $f$.
\end{proof}

For homotopy classes we have to add the `no-disk' condition again:
\begin{lemma} Let $\alpha$ and $\beta$ be free homotopy classes of curves that have no disk in their complement. Then $\alpha$ and $\beta$ lie in the same mapping class group orbit if and only if $\Gamma(\alpha)\simeq\Gamma(\beta)$ and the data $(P(\beta),S(\beta))$ can be obtained from $(P(\alpha),S(\alpha))$ by applying the isomorphism between these graphs to this data.
\end{lemma}
\begin{proof}
Consider two minimal representatives $a$ and $b$ of $\alpha$ and $\beta$; by Theorem \ref{reidemeisterthm}, they are related by an isotopy and a sequence of Reidemeister moves. But since there is no disk in the complement, there is no Reidemeister move that can be performed, so $a$ and $b$ are actually isotopic and we can apply Lemma \ref{lem_orbitclass1}.
\end{proof}

\section{Counting ribbon graph embeddings}

\subsection{Set up}
Our ultimate goal is to understand the asymptotics of $N_{g,n}(k,h)$, the number of $\mcg(\Sigma_{g,n})$-orbits of free homotopy classes of curves with $k$ self-intersections and minimal genus $h$ on $\Sigma_{g,n}$. Note that the minimal genus of a curve is actually the genus of the ribbon surface associated to it. In this section we will give upper and lower estimates on $N_{g,n}(k,h)$.

Given a ribbon graph $\Gamma\in \cg(k)$, we define the following:
\begin{enumerate}
\item The number $N_{g,n}^{iso}(\Gamma)$ of embeddings of $\Gamma$ into $\Sigma_{g,n}$ up to isotopy.
\item The number $N^{\circ}_{g,n}(\Gamma)$ of embeddings of $\Gamma$ into $\Sigma_{g,n}$ with no disk in the complement up to homotopy
\end{enumerate}
Furthermore, we define the set of ribbon graphs of genus $h$ (that correspond to single curves) with $k$ vertices as
$$
\cg_h(k) = \{\Gamma\in\cg(k) |\; g(\Gamma)=h \}.
$$
For most $h$, the set above is empty. In fact, an Euler characteristic tells us that if $\Gamma$ is ribbon graph corresponding to a curve with $k$ self-intersections, we have $k+1-2g(\Gamma)\geq 0$.

With the notation above, we have
$$\sum_{\Gamma\in \cg_h(k)} N^{\circ}_{g,n}(\Gamma)\leq N_{g,n}(k,h) \leq \sum_{\Gamma\in \cg_h(k)} N_{g,n}^{iso}(\Gamma).$$

Note that $\sum\limits_{\Gamma\in\cg_h(k)} N_{g,n}^{iso}(\Gamma)$ overcounts $N_{g,n}(k,h)$ in two ways:
\begin{itemize}
\item we count isotopy classes instead of homotopy classes of curves, so, because of third Reidemeister moves, multiple isotopy classes might correspond to the same homotopy class;
\item if we glue disks to some component of the ribbon surface associated to $\Gamma$ to obtain $\Sigma_{g,n}$, the curve corresponding to $\Gamma$ may be not in minimal position on $\Sigma_{g,n}$.
\end{itemize}
On the other hand, there are minimal generic curves $a$ such that one or more component of $\Sigma_{g,n}\setminus \Gamma(a)$ is a disk, so in general the first inequality is strict.

\subsection{Estimating $N^{iso}_{g,n}(\Gamma)$ and $N^{\circ}_{g,n}(\Gamma)$}

We will count the total number of distinct (up to homeomorphism and isotopy) ways to embed the ribbon surface corresponding to $\Gamma$ on $\Sigma_{g,n}$. Recall that $N_{g,n}^{iso}(\Gamma)$ counts all such embeddings, including non-minimal ones.  

To shorten notation, we will let $(g_0,b_0)=(g(\Gamma),b(\Gamma))$ denote the signature of the ribbon surface corresponding to $\Gamma$. Furthermore, $\Sigma_{g,b,n}$ will denote the toplogical surface of genus $g$ with $b$ boundary components and $n$ punctures.

\begin{lemma}  Suppose we have an embedding of the ribbon surface $\Gamma$ in $\Sigma_{g,n}$, where
$$\Sigma_{g,n}\setminus \Gamma=\bigsqcup_{i=1}^r \Sigma_{g_i,b_i,n_i}.$$
Then 
$$\sum_{i=1}^r n_i=n ,\;\sum_{i=1}^r b_i=b_0 \text{ and } \sum_{i=1}^r g_i = g+r-g_0-b_0.$$ \label{conditions}
\end{lemma}

\begin{proof}  The first identity comes from the fact that $\Gamma$ contains no punctures. The second comes from the fact that the boundaries of the $\Sigma_{g_i,b_i,n_i}$ are glued to $\Gamma$.  For the third identity we have
$$\chi(\Sigma_{g,n})=\chi(\Gamma)+\sum_{i=1}^r\chi(\Sigma_{g_i,b_i,n_i}).$$

From this we obtain
$$2-2g-n=2-2g_0-b_0+\sum_{i=1}^r (2-2g_i-b_i-n_i),$$
combining this with the first identity gives
$$g+r-g_0-b_0=\sum_{i=1}^r g_i.$$
\end{proof}

Next we turn to counting the number of ways to embed the ribbon surface $\Gamma$ in $\Sigma_{g,n}$.

\begin{prp}\label{prp_uboundN} Let $\{g_i,n_i\}_{i\in\mathbb{N}}\subset \mathbb{N}$ be a sequence such that $g_i+n_i\to\infty$ as $i\to\infty$. Then
$$
N^{iso}_{g_i,n_i}(\Gamma) = \frac{1}{\card{\baut{\Gamma}}}  {g_i+b_0-g_0-1 \choose b_0-1} {n_i+b_0-1 \choose b_0-1} + O\left(((g_i+1)(n_i+1))^{b_0-2}\right)
$$
as $i\to\infty$.
\end{prp}

\begin{proof} We first assume that the boundary components of $\Gamma$ are distinguishable. By this we mean that $\baut{\Gamma}=\{\mbox{Id}\}$.

Recall that $B=B(\Gamma)=\{\beta_1,\ldots,\beta_{b_0} \}$ denotes the set of boundary components of $\Gamma$. Because of our assumption on $\baut{\Gamma}$, we can write
$$
N_{g_i,n_i}^{iso}(\Gamma) = \sum_{r=1}^{b_0}\sum_{\substack{P \models B \\ \card{P}={r}}} N_{g_i,n_i}^{iso}(\Gamma,P)
$$
where the notation $P\models B$ means that $P=\{p_1,\ldots,p_r\}$ is a set partition of $B$. The number $N_{g_i,n_i}^{iso}(\Gamma,P)$ counts the embeddings of $\Gamma$ into $\Sigma_{g_i,n_i}$ such that the boundary components of in $p_j$ form the entire boundary of a single connected component $S_j$ in $\Sigma_{g_i,n_i}\setminus \Gamma$ for all $j=1,\ldots, n$.

Lemma \ref{lem_orbitclass1} tells us that two such embeddings corresponding to a set partition $P$ are homeomorphic if and only if the signatures of all the $S_i$ are the same (we again use our assumption that $\Gamma$ has no boundary permuting automorphisms here). This means that $N_{g_i,n_i}^{iso}(\Gamma,P)$ is equal to the number of ways to distribute the total genus  $g_i+r-g_0-b_0$ and number of punctures $n_i$ over the (distinguishable) $r$ subsets in $P$. The number of ways to distribute a number $s$ over $r$ boxes is also called the number of {\it weak compositions of $s$ into $r$ parts}. As such, we obtain
$$
N_{g_i,n_i}^{iso}(\Gamma,P) = {g_i+2r-g_0-b_0-1 \choose r-1} {n_i+r-1 \choose r-1}
$$
(see for instance \cite[p.15]{Stanley}). Because $N_{g_i,n_i}^{iso}(\Gamma,P)$ only depends on the number of parts $r$ of the partition, we obtain
$$
N_{g_i,n_i}^{iso}(\Gamma) = \sum_{r=1}^{b_0}S(b_0,r){g_i+2r-g_0-b_0-1\choose r-1 } {n_i+r-1\choose r-1 }
$$
where $S(b_0,r)$ is a Stirling number of the second kind, which counts the number of set partitions of $B$ into $r$ parts (see for instance \cite[p.33]{Stanley}). 

The sum above is a finite sum in our considerations: $\Gamma$ is fixed, hence so is $b_0$. This means that the terms that contribute to the asymptotics are only those of highest order in $g_i+n_i$.

For fixed $r$ we have that
$$
{s+r-1 \choose r-1} \sim \frac{s^{r-1}}{(r-1)!}
$$
as $s\to\infty$. This means that 
$$
N^{iso}_{g_i,n_i}(\Gamma) = S(b_0,b_0) {g_i+b_0-g_0-1 \choose b_0-1} {n_i+b_0-1 \choose b_0-1} + O\left(((g_i+1)(n_i+1))^{b_0-2}\right)
$$
as $i\to\infty$. Because $S(b_0,b_0)=1$, this gives us the result in the case where $\baut{\Gamma}=\{\mbox{Id}\}$.

In the case where $\Gamma$ does have automorphisms that permute boundary components then we over-count. For arbitrary $P\models B$, it is quite hard to work out the influence of the automorphisms. However, we are lucky and the only embeddings that contribute to the asymptotics are those for which the ribbon graph disconnects $\Sigma_{g_i,n_i}$ into $b_0$ surfaces. 

In fact, we will prove that we can also ignore those embeddings in which some of these surfaces have the same signature. Note that in this case this is equivalent to having the same genus and number of punctures, since in these embeddings every connected component of $\Sigma_{g_i,n_i}\setminus \Gamma$ has exactly one boundary component. Once we have this the proof is done, because every embedding in which the genera of all the complementary surfaces are different is counted exactly $\card{\baut{\Gamma}}$ times.

Let us denote by $N_{g_i,n_i}^{iso}(\Gamma,\mathrm{rep})$ the number of gluings corresponding to the set partition $P$ with $b_0$ elements in which at least two of the signatures are equal. We claim that
$$
N_{g_i,n_i}^{iso}(\Gamma,\mathrm{rep}) = O\left(((g_i+1)(n_i+1))^{b_0-2}\right)
$$
as $i\to\infty$. In fact, this follows from a simple union type bound. Indeed, every gluing that contributes to $N_{g_i,n_i}^{iso}(\Gamma,\mathrm{rep})$ can be obtained by choosing a pair of boundary components of $\Gamma$, assigning a single genus and number of punctures to those two and then assigning genera to all the other boundary components. This means that we can bound $N_{g_i,n_i}^{iso}(\Gamma,\mathrm{rep})$ as
$$
N_{g_i,n_i}^{iso}(\Gamma,\mathrm{rep}) \leq {b_0 \choose 2} (g_i+1)\cdot (g_i+1)^{b_0-3}\cdot( n_i+1) \cdot (n_i+1)^{b_0-3}
$$
where the power $b_0-3$ comes from the fact that once the genera (or numbers of punctures) of the first $b_0-3$ boundary components are chosen, the genus (or number of punctures) of the last boundary component is fixed. This proves the claim.
\end{proof}

Towards our lower bound we obtain the following proposition:

\begin{prp}\label{prp_lboundN}  Let $\{g_i,n_i\}_{i\in\mathbb{N}}\subset \mathbb{N}$ be a sequence such that $g_i+n_i\to\infty$ as $i\to\infty$. Then
$$
N_{g_i,n_i}^{iso}(\Gamma)-N_{g_i,n_i}^\circ(\Gamma) = O(((g_i+1)(n_i+1))^{b_0-2})
$$
as $i\to \infty$.
\end{prp}

\begin{proof} If we were to use the count in the proof of Proposition \ref{prp_uboundN} for $N_{g_i,n_i}^\circ(\Gamma)$, we would overcount because we did not worry about attaching disks. Note that we use Theorem \ref{reidemeisterthm} here: if there are no (unpunctured) disks in the complement there is only one isotopy class in each homotopy class. In order to obtain a lower bound $N_{g_i,n_i}^\circ(\Gamma)$ we will simply subtract the number of gluings which attach a disk. Let us call this number $N_{g_i,n_i}(\Gamma,\mathrm{Disk})$. So
$$
N_{g_i,n_i}^{iso}(\Gamma)-N_{g_i,n_i}^\circ(\Gamma) = N_{g_i,n_i}(\Gamma,\mathrm{Disk})
$$

We have
$$
N_{g_i,n_i}(\Gamma,\mathrm{Disk}) \leq \sum_{j=1}^{b_0} N_{g_i,n_i}(\Gamma,\mathrm{Disk},j)
$$
where $N_{g_i,n_i}(\Gamma,\mathrm{Disk},j)$ counts the number of gluings in which an unpunctured disk is attached to the $j^{\text{th}}$ boundary component $\beta_j$ (and possibly also to some of the other boundary components). Because we are only after a bound on $N_{g_i,n_i}(\Gamma,\mathrm{Disk},j)$, we will disregard the influence of automorphisms. Using the exact same arguments as in Proposition \ref{prp_uboundN}, we obtain:
\begin{gather*}
N_{g_i,n_i}(\Gamma,\mathrm{Disk},j) \leq \sum_{r=1}^{b_0-1} S(b_0-1,r)  {g_i+2r-g_0-b_0-1\choose r - 1 } {n_i+r-1 \choose r-1} =\\
= O\left(((g_i+1)(n_i+1))^{b_0-2}\right)
\end{gather*}
as $i\to\infty$. 
\end{proof}

\section{The main theorem} \label{sec_asymptotics}

\subsection{Counting orbits}
We are now ready to determine the asymptotics of $N_{g,n}(k,h)$. Before we state our result, we define
$$
C_{k,h}= \sum_{\Gamma \in \cg_h(k)} \frac{1}{\card{\baut{\Gamma}}}
$$
Note that this is a constant in all our considerations. We will write $C_k=C_{k,0}$ for the constant corresponding to planar ribbon graphs.

We have the following result:
\begin{thmrep}{\ref{maintheorem}} Let $k,h\in\mathbb{N}$. Furthermore, let $\{g_i,n_i\}_{i\in\mathbb{N}}\subset \mathbb{N}$ be a sequence such that $g_i+n_i\to\infty$ as $i\to\infty$. Then
$$
 N_{g_i,n_i}(k,h) \sim C_{k,h} {g_i+k-3h+1 \choose k+1-2h} {n_i+k+1-2h \choose k+1-2h} 
$$
as $i\to\infty$.
\end{thmrep}
\begin{proof} We will of course use the bounds from the previous section. Propositions \ref{prp_uboundN} and \ref{prp_lboundN} imply that $N^{\circ}_{g_i,n_i}(\Gamma)\sim N^{iso}_{g_i,n_i}(\Gamma)$ as $i\to\infty$. So we obtain that
$$
 N_{g_i,n_i}(k,h) \sim \sum_{\Gamma \in \cg_h(k)} \frac{1}{\card{\baut{\Gamma}}}  {g_i+b(\Gamma)-h-1 \choose b(\Gamma)-1} {n_i+b(\Gamma)-1 \choose b(\Gamma)-1} 
$$
as $i\to\infty$. A simple Euler characteristic argument yields that for $\Gamma\in\cg_h(k)$ we have
$$
b(\Gamma) =  k+2-2h.
$$
Hence we obtain
$$
 N_{g_i,n_i}(k,h) \sim C_{k,h}  {g_i+k-3h+1 \choose k+1-2h} {n_i+k+1-2h \choose k+1-2h} 
$$
as $i\to\infty$.
\end{proof}

Note that simple curves technically do not fall within our scope, because to construct a ribbon graph, we need self-intersections. However, using the annulus as the single ribbon surface corresponding to a simple curve, all the arguments above work. As a ribbon surface, the annulus has one automorphism, permuting the two boundary components.

As a consequence of our main theorem we also obtain the asymptotics of the number of orbits of all curves with $k$ self-intersections.
\begin{correp}{\ref{cor_nogenus}} Let $k\in\mathbb{N}$. Furthermore, let $\{g_i,n_i\}_{i\in\mathbb{N}}\subset \mathbb{N}$ be a sequence such that $g_i+n_i\to\infty$ as $i\to\infty$. 
$$
 N_{g_i,n_i}(k) \sim C_{k}  {g_i+k+1 \choose k+1} {n_i+k+1 \choose k+1} 
$$
as $i\to\infty$.
\end{correp}

\begin{proof} We have
$$
N_{g_i,n_i}(k) = \sum_{h=0}^{\lfloor{\frac{k+1}{2}}\rfloor} N_{g_i,n_i}(k,h)
$$
Theorem \ref{maintheorem} tells us that asymptotically only the term corresponding to $h=0$ contributes to the sum above, which yields the corollary.
\end{proof}

\subsection{Ribbon Graph Automorphisms and $C_k$}
Now we briefly discuss the constant $C_{k}= \sum\limits_{\Gamma \in \cg_0(k)} \frac{1}{\card{\baut{\Gamma}}}$.  Cantarella, Chapman and Mastin \cite{CantarellaChapmanMastin} recently enumerated planar ribbon curve graphs (knot diagram shadows on the sphere) with 10 or fewer crossings, as well as the mean number of automorphisms of such a shadow.  We are not interested in the total number of automorphisms, but rather automorphisms inducing distinct permutations on the set of boundary components of the ribbon graph.  But of course, if the ribbon graph has no automorphisms, there are no boundary automorphisms either.  Cantarella, Chapman and Mastin's results show that the total number of planar ribbon curve graph automorphisms decreases rapidly and is already $1.03$ for $k=10$, so we expect that as $k\to \infty$, $C_k\sim |\cg_0 (k) | $.  

For small $k$, one can work out the constant $C_{k}$ explicitly. The following table lists the first four values.

\begin{table}[H]
\begin{tabular}{c|c c c c}
$k$& 0& 1 & 2 & 3\\
\hline 
\rule{0pt}{3ex}$C_k$& $\frac{1}{2}$ & $\frac{1}{2}$ & $\frac{1}{2}$ & $3$
\end{tabular}
\caption{The first four values of $C_k$.}
\end{table}

Quantities similar to $|\cg(k)|$ and $|\cg_h(k)|$ for fixed $h$ have been studied by several authors.  One can show (see for example Turaev \cite{Turaev}) that elements of $\cg(k)$ are in bijection with Gauss diagrams with one core circle and $k$ arrows, or in Turaev's language, virtual strings.  Chord diagrams are Gauss diagrams with unoriented arrows.  The asymptotics of the number of chord diagrams with $k$ arrows, with either oriented or unoriented core circles, were studied by Stoimenow \cite{Stoimenow}, giving a lower bound for $|\cg(k)|$, though his results are not filtered by genus. 

\subsection{Probabilistic Statements} Our reasoning also allows us to make certain probabilistic statements. Because the set of $\mcg(\Sigma_{g,n})$-orbits of curves with $k$ self-intersections is finite, it carries a natural probability measure, coming from the counting measure. We will denote this measure by $\mathbb{P}_{g,n}$.

Along the way we have proved the following:
\begin{correp}{\ref{cor_prob2}} Let $k\in\mathbb{N}$. Furthermore, let $\{g_i,n_i\}_{i\in\mathbb{N}}\subset \mathbb{N}$ be a sequence such that $g_i+n_i\to\infty$ as $i\to\infty$. Then as $i\to\infty$
\[ \mathbb{P}_{g_i,n_i}\left[\substack{\displaystyle{\text{A curve with }k\text{ self-intersections has a disk}} \\ \displaystyle{\text{in the complement}}} \right]\to 0 \]
and
\[ \mathbb{P}_{g_i,n_i}\left[\substack{\displaystyle{\text{A curve with }k\text{ self-intersections separates }\Sigma_{g_i,n_i}\text{ into}} \\ \displaystyle{k+2\text{ surfaces, all of different signatures}}} \right]\to 1 \]
\end{correp}

Furthermore, it also follows from our arguments that if we let $N^{iso}_{g,k}$ denote the number of free {\it isotopy} classes of essential closed curves with $k$ self-intersections, then:
$$
N^{iso}_{g,k} \sim N_{g,k}
$$
for $k$ fixed and $g+n\to\infty$. This implies a probablistic version of Baer's theorem:
\begin{correp}{\ref{cor_prob1}}  Let $k\in\mathbb{N}$. Furthermore, let $\{g_i,n_i\}_{i\in\mathbb{N}}\subset \mathbb{N}$ be a sequence such that $g_i+n_i\to\infty$ as $i\to\infty$. Then as $i\to\infty$
$$\mathbb{P}_{g_i,n_i}\left[\substack{\displaystyle{\text{The homotopy class of a curve with }k\text{ self}} \\ \displaystyle{\text{intersections contains exactly one isotopy class}}} \right] \to 1$$
\end{correp}

\section{Geometric consequences} \label{sec_geometry}

In this section we prove Corollary \ref{cor_geometry}. This will be a direct consequence of Theorem \ref{maintheorem} and the following result by Basmajian:

\begin{theor}[\cite{Basmajian}, Theorems 1.1 and 1.2]\label{thm_Basmajian} Let $X$ be a complete hyperbolic structure on $\Sigma_{g,n}$ and $\gamma$ a geodesic on $X$ with $k\geq 1$ self-intersections. The length of $\gamma$ on $X$ satisfies
$$
\max\{c_X\sqrt{k},\frac{1}{4}\log(2k)\} \leq \ell_X(\gamma) 
$$
where $c_X=0$ if $X$ has cusps and $c_X$ is a continuous function on the moduli space of hyperbolic structures on $\Sigma_{g}$, tending to $0$ as $X$ approaches the boundary of this moduli space.
\end{theor}

Given $L>0$ and a hyperbolic surface $X$, define
$$
a_X(L)=\left\lfloor \min \left\{ \left(\frac{L}{c_X}\right)^2,\frac{1}{2}e^{4L} \right\} \right\rfloor +1
$$
Note that for fixed $L$, $a_\cdot(L)$ is a uniformly bounded function on the set of hyperbolic surfaces. Namely, 
$$
A(L)=\sup_X\{a_X(L)\} = \left\lfloor\frac{1}{2}e^{4L}\right\rfloor +1 <\infty
$$
where the supremum is to be taken over all hyperbolic surfaces $X$ of all possible genera.

Recall that $N_X(L)$ denotes the number of $\mcg(\Sigma_{g(X),n(X)})$-orbits of closed geodesics on $X$ that contain a curve of length at most $L$.

We are now ready to prove the following:
\begin{correp}{\ref{cor_geometry}} Let $L>0$. There exists a constant $C=C(L)>0$ such that for any hyperbolic surface $X$
$$
N_X(L) \leq C\cdot ((g(X)+1) (n(X)+1))^{a_X(L)}
$$
\end{correp}

\begin{proof} Theorem \ref{thm_Basmajian} tells us that an $\mcg(\Sigma_{g,n})$-orbit can only contain a curve of length $\leq L$ on $X$ if it's an orbit of curves with at most $a_X(L)-1$ self-intersections. This means that: 
$$
N_X(L) \leq \sum_{k=0}^{a_X(L)-1}N_{g(X),n(X)}(k)
$$
Theorem \ref{maintheorem} tells us that asymptotically this sum is dominated by its last term and that furthermore there exists a constant $C=C(L)$ such that
$$
\sum_{k=0}^{a_X(L)-1}N_{g(X),n(X)}(k) \leq C\cdot (g(X)+1)^{a_X(L)}\cdot (n(X)+1)^{a_X(L)}
$$
where we have used the fact that $a_X(L)\geq 1$.
\end{proof}

Note that the fact that $a_X(L)$ is uniformly bounded for fixed $L$ gives us a polynomial upper bound. On the other hand, if something is known about the hyperbolic structure and $c_X$ can be controlled, then this bound becomes sharper.

\bibliographystyle{plain}
\bibliography{referencescfp}

\end{document}